\documentclass[reqno,11pt]{amsart}

\pdfoutput=1

\usepackage{pinlabel}
\input{epsf.tex}
\usepackage{latexsym,amsfonts,amssymb,verbatim,mathrsfs,amsthm}
\usepackage{amsmath,amsthm,amssymb,latexsym,graphics,textcomp}
\usepackage{eucal,eufrak}
\usepackage{colonequals}
\usepackage{graphicx}
\usepackage{url}
\input{xy}
\xyoption{all}

\theoremstyle{plain}
\newtheorem{theorem}{Theorem}[section]
\newtheorem{lemma}[theorem]{Lemma}

\newtheorem{definition}[theorem]{Definition}

\newtheorem*{claim}{Claim}

\theoremstyle{definition}
\newtheorem{remark}{Remark}
\newtheorem{question}{Question}

\def\1{\mathbf 1}

\def\Mod{\mathrm{Mod}}

\def\Aut{\mathrm{Aut}}

\def\C{\mathcal C}

\def\XX{\mathfrak X}

\def\CC{\mathfrak C}
\def\OO{\mathfrak O}
\def\UU{\mathfrak U}
\def\BB{\mathfrak B}
\def\SS{\mathfrak S}

\title{Finite rigid sets in curve complexes}
\author{Javier Aramayona and Christopher J. Leininger}
\begin{document}

\begin{abstract}
We prove that curve complexes of surfaces are {\em finitely rigid}:  for every orientable surface $S$ of finite topological type, we identify a finite subcomplex $\XX$ of the curve complex $\mathcal{C}(S)$ such that every locally injective simplicial map $\XX\to \mathcal{C}(S)$
is the restriction of an element of $\Aut(\C(S))$, unique up to the (finite) point-wise stabilizer of $\XX$ in $\Aut(\C(S))$.  Furthermore, if $S$ is not a twice-punctured torus, then we can replace $\Aut(\C(S))$ in this statement with the extended mapping class group $\Mod^\pm(S)$.
\end{abstract}

\maketitle

\section{Introduction}

The curve complex $\C(S)$ of a surface $S$ is a simplicial complex whose $k$-simplices correspond to sets of $k+1$ distinct isotopy classes of essential simple closed curves  on $S$ with pairwise disjoint representatives. A celebrated theorem of Ivanov \cite{Ivanov}, Korkmaz \cite{Korkmaz} and Luo \cite{Luo} asserts that curve complexes are {\em simplicially rigid}: the group $\Aut(\C(S))$ of  simplicial automorphisms of $\C(S)$ is, except in a few well-understood cases, isomorphic to the {\em extended mapping class group} $\Mod^{\pm}(S)$. 
This result has subsequently been extended by Irmak \cite{Irmak},  Behrstock-Margalit \cite{Behrstock-Margalit}, and Shackleton \cite{Shackleton} to more general types of simplicial self-maps of $\C(S)$, such as {\em superinjective} and {\em locally injective} maps (recall that a simplicial map is locally injective if the restriction to the star of every vertex is injective).

In this article, we prove that curve complexes are {\em finitely rigid}, answering a question of Lars Louder \cite{Louder}. More concretely, we will show:

\begin{theorem} \label{T:main}
Let $S$ be an orientable surface of finite topological type. Then there exists a finite simplicial complex $\XX \subset \C(S)$ such that for any locally injective simplicial map
\[ \phi:\XX \to \C(S) \]
there exists an element $f \in \Aut(\C(S))$ with $\phi = f|_{\XX}$.  Moreover, $f$ is unique up to the (finite) point-wise stabilizer $H_{\XX} < \Aut(\C(S))$ of $\XX$.
\end{theorem}

We will call the set $\XX$ in Theorem \ref{T:main} a {\em finite rigid set} in $\C(S)$ (throughout, we will confuse subcomplexes $\XX \subset \C(S)$ and their vertex sets).
The finite rigid sets we will construct in the proof of Theorem \ref{T:main} have diameter at most 2 in $\C(S)$. Therefore, a natural question is:

\begin{question}
Are there finite rigid sets in $\C(S)$ of arbitrarily large diameter?
\end{question}

As we explain in Section \ref{S:spheresofcurves}, if $S$ is a sphere with punctures, the finite rigid set $\XX$ that we construct is a simplicial sphere of dimension equal to the homological dimension of $\C(S)$. In fact, if $S$ is a sphere with six punctures, then $\XX$ is the sphere considered by Broaddus in \cite{Broaddus} where he showed that it generates the top-dimensional homology group of $\C(S)$. We thus ask the following question; compare with  Conjecture 4.9 of \cite{Broaddus}:

\begin{question}
Suppose $S$ is a sphere with punctures. Are the finite rigid sets in $\C(S)$ constructed in Section \ref{S:spheres} homologically non-trivial?
\label{Q:homology}
\end{question}

The proof of Theorem \ref{T:main} is entirely constructive, although our methods become increasingly complicated as the genus of $S$ grows. For this reason, we will first establish Theorem \ref{T:main} for $S$ of genus 0 (Theorem \ref{T:sphere}). In Theorem \ref{T:tori} we will prove our main result in the case of $S$ a torus. We will then treat the case of $S$ of genus at least 2, dealing first with closed surfaces (Theorem \ref{T:closedsurfaces}) and then with the general case (Theorem \ref{T:punctured}).   We note that Theorems \ref{T:sphere}, \ref{T:tori}, \ref{T:closedsurfaces} and \ref{T:punctured} contain more precise statements involving $\Mod(S)$ or $\Mod^\pm(S)$.  Theorem \ref{T:main} then follows from an application of the results in \cite{Ivanov,Korkmaz,Luo}, which state that the homomorphism $\Mod^\pm(S) \to \Aut(\C(S))$ is surjective, with finite kernel, except when $S$ is a twice punctured torus.

The techniques we use are similar to those used in \cite{Ivanov,Korkmaz,Luo,Behrstock-Margalit,Irmak,Shackleton}.              
A common strategy is to use only intersection/nonintersection properties about
some set of curves to deduce more precise geometric information about a particular pair of
curves.  For example, one can deduce that a pair of curves intersects once.     
Given an automorphism of the curve complex, one uses this information to
construct a candidate homeomorphism of the surface.  It then   
remains to verify that the automorphism induced by the candidate
homeomorphism agrees with the given automorphism on every curve.

We use this same strategy, but must carefully choose the finite set of    
curves to find the right balance: not enough curves, and we cannot
determine the candidate homeomorphism; too many curves, and we cannot keep
enough control to show that the candidate induces the given automorphism
on every curve in our set.

\medskip

\noindent {\bf Non-examples.}  One might guess that a set of curves like the one shown on the left of Figure \ref{F:nonexample} would provide a finite rigid set $\XX \subset \C(S)$.  However, sending the curves on the left to the ones on the right determines a simplicial injection of $\XX$ into $\C(S)$ which cannot come from the restriction of an automorphism of $\C(S)$.  We also note that on a closed surface of genus $g \geq 2$, there is no finite rigid set $\XX$ with less than $3g-2$ curves, since any such set can be mapped to some subset of the curves in {\em any} pants decomposition.

\begin{figure}[htb]
\begin{center}
\includegraphics[height=1in]{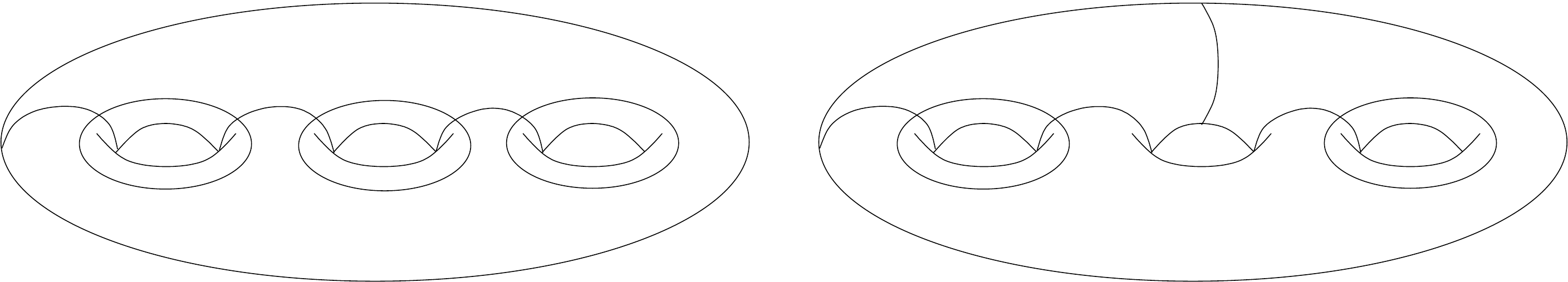} \caption{A finite non-rigid set on the left, and its image under a simplicial embedding on the right.} \label{F:nonexample}
\end{center}
\end{figure}

\noindent{\bf Acknowledgements.} The authors thank Lars Louder for asking the question that motivated this paper, and Dan Margalit for helpful conversations.

\section{Definitions}
\label{S:defs}

Let $S=S_{g,n}$ be an orientable surface of genus $g$ with $n$ punctures, and define the {\em complexity} of $S$ as $\xi(S) =3g-3+n$. We say that a simple closed curve on $S$ is {\em essential} if it does not bound a disk or a once-punctured disk on $S$.  An {\em essential subsurface} of $S$ is a properly embedded subsurface $N \subset S$ for which each boundary component is an essential curve in $S$.  By a {\em hole} in a (sub)surface we will mean either a boundary component or puncture (and when the distinction is irrelevant, we will not distinguish between a puncture and a boundary component).

The {\em curve complex} $\C(S)$ of $S$ is a simplicial complex whose $k$-simplices correspond to sets of $k+1$ isotopy classes of essential simple closed curves on $S$ with pairwise disjoint representatives.   In order to simplify the notation, a set of isotopy classes of simple closed curves will be confused with its representative curves, the corresponding vertices of $\C(S)$, and the subcomplex of $\C(S)$ spanned by the vertices.  We also assume that representatives of isotopy classes of curves and subsurfaces intersect essentially (that is, transversely and in the minimal number of components).

If $\xi(S) > 1$, then $\C(S)$ is a connected complex of dimension $\xi(S)-1$. If $\xi(S)\le 0$ and $S\ne S_{1,0}$, then $\C(S)$ is empty. If $\xi(S)=1$ or $S=S_{1,0}$, then $\C(S)$ is a countable set of vertices; in order to obtain a connected complex, we modify the definition of $\C(S)$ by declaring $\alpha, \beta\in \C^{(0)}(S)$ to be adjacent in $\C(S)$ whenever $i(\alpha,\beta)=1$ if $S=S_{1,1}$ or $S=S_{1,0}$, and whenever $i(\alpha,\beta)=2$ if $S=S_{0,4}$.  Furthermore, we add triangles to make $\C(S)$ into a flag complex.   In all three cases, the complex $\C(S)$ so obtained is isomorphic to the well-known {\em Farey complex}.

Recall that a {\em pants decomposition} $P$ is a simplex in $\C(S)$ of dimension $\xi(S) - 1$;  equivalently, $P$ is a set of pairwise disjoint curves whose complement in $S$ is a disjoint union of three-holed spheres.

\begin{definition}
Let $S$ be a surface, and $A$ a set of simple closed curves on $S$. We say that $A$ is {\em filling (in $S$)} if for any $\beta \in \C(S)$, $i(\alpha,\beta) \neq 0$ for some $\alpha \in A$.  Say that $A$ is {\em almost-filling (in $S$)} if there are only finitely many curves $\beta \in \C(S)$ such that $i(\alpha, \beta)=0$ for all $\alpha \in A$. If $A$ is almost-filling, then the set 
\[ B = \{ \beta \in \C^{(0)}(S) \setminus A \mid i(\beta,\alpha) = 0, \, \forall \alpha \in A  \} \] 
is called the set of curves {\em determined by $A$.} In the special case when $B=\{\beta\}$, we say that the simple closed curve $\beta$ is {\em uniquely determined} by $A$. 
\end{definition}

An example of an almost filling set to keep in mind is the following.  Suppose $N \subset S$ is an essential subsurface, $A$ is set of curves which is filling in $N$, and $P$ is a pants decomposition of $S - N$ (not including $\partial N$).  Then $A \cup P$ is an almost filling set, and $\partial N$ is the set of curves determined by $A \cup P$.

\begin{definition}  Let $S$ be a surface and $\XX \subset \C(S)$ a subcomplex.  If $\alpha,\beta \in \XX$ are curves with $i(\alpha,\beta) \neq 0$, then we say that their intersection is {\em $\XX$--detectable} (or simply {\em detectable} if $\XX$ is understood) if there are two pants decompositions $P_\alpha,P_\beta \subset \XX$ such that
\begin{equation} \label{E:detectable} \alpha \in P_\alpha, \, \beta \in P_\beta, \mbox{ and } P_\alpha - \alpha = P_\beta - \beta. \end{equation}
\label{D:detectable}
\end{definition}
We note that if $\alpha,\beta$ have detectable intersection, then they must fill a $\xi=1$ (essential) subsurface, which we denote $ N(\alpha\cup\beta) \subset S$.   For notational purposes, we call $P = P_\alpha - \alpha = P_\beta - \beta$ a {\em pants decomposition of $S - N(\alpha\cup\beta)$}, even though it includes the boundary components of $N(\alpha\cup\beta)$.
\begin{lemma} \label{L:detectable2detectable}
Let $\XX \subset \C(S)$ be a subcomplex, and $\alpha,\beta \in \XX$ intersecting curves with $\XX$--detectable intersection.  If $\phi:\XX \to \C(S)$ is a locally injective simplicial map, then $\phi(\alpha),\phi(\beta)$ have $\phi(\XX)$--detectable intersection, and hence fill a $\xi=1$ subsurface.
\end{lemma}
\begin{proof}
Let $P_\alpha$ and $P_\beta$ the pants decompositions as in \eqref{E:detectable}. Since $\phi$ is locally injective and simplicial, it follows that $\phi(P_\alpha)$ and $\phi(P_\beta)$ are pants decompositions. Moreover,
\[ \phi(\alpha) \in \phi(P_\alpha), \, \phi(\beta) \in \phi(P_\beta), \mbox{ and } \phi(P_\alpha) - \phi(\alpha) =\phi( P_\beta) - \phi(\beta), \]
and hence $\phi(\alpha)$ and $\phi(\beta)$ have $\phi(\XX)$-detectable intersection. 
\end{proof}

\begin{definition}   
Let  $\alpha$ and $\beta$ be curves on $S$ which fill a $\xi=1$ subsurface $N \subset S$.  We say $\alpha$ and $\beta$ are {\em Farey neighbors} if they are adjacent in $\C(N)$.   We say that $\alpha$ and $\beta$ are {\em nearly Farey neighbors} if they are {\bf not} Farey neighbors, but $\alpha = T_\gamma^{\pm 1}(\beta)$, where $T_\gamma$ is a Dehn twist in a curve $\gamma$ which is a Farey neighbor of both $\alpha$ and $\beta$.
\end{definition}

We note that if $\alpha$ and $\beta$ are nearly Farey neighbors, then $N$ must be a four-holed sphere.

One should compare Lemma \ref{L:detectable2detectable} to \cite[Lemma 6]{Shackleton} and the notion of Farey neighbors with that of ``small intersection'' of \cite{Shackleton}.  
We also record the following observation and note the similarity with \cite[Lemma 1]{Ivanov}.

\begin{lemma} \label{L:minimallyfilling} Suppose $\alpha_1,\alpha_2,\alpha_3,\alpha_4$ are distinct curves on $S$, with $\alpha_2,\alpha_3$ filling a $\xi=1$ subsurface and 
\[ i(\alpha_i,\alpha_j) = 0 \Leftrightarrow |i-j| > 1\]
for all $i \neq j$.  Then $\alpha_2,\alpha_3$ are Farey neighbors or nearly Farey neighbors.

In addition, if $\alpha_1,\alpha_2,\alpha_3,\alpha_4 \in \XX \subset \C(S)$  and all intersections among the $\alpha_i$ are $\XX$--detectable, then for any locally injective simplicial map $\phi:\XX \to \C(S)$, the curves $\phi(\alpha_2),\phi(\alpha_3)$ are Farey neighbors or nearly Farey neighbors.
\end{lemma}
\begin{proof}
The second part follows from the first and Lemma \ref{L:detectable2detectable} so we need only prove the first part.
The assumption on intersections implies that there are essential arcs $\delta_i \subset \alpha_i \cap N(\alpha_2\cup\alpha_3)$, for $i \in \{1,4\}$, with $\delta_1,\delta_4$  disjoint.  The curves $\alpha_2$ and $\alpha_3$ are the unique curves in $N(\alpha_2\cup\alpha_3)$ disjoint from $\delta_4$ and $\delta_1$, respectively, and since they fill $N(\alpha_2 \cup \alpha_3)$, by inspecting the possible configurations, we see that they are Farey neighbors or nearly Farey neighbors.
\end{proof}

In the previous proof, we note that up to homeomorphism, there is only one configuration of the arcs $\delta_1,\delta_4$ in $N$ which makes $\alpha_2,\alpha_3$ into nearly Farey neighbors.  Namely, if $N$ is a four-holed sphere and $\delta_1$ has endpoints on the same two distinct boundary components as $\delta_4$ (for example, consider a curve $\gamma$ intersecting $\delta_1$ in a single point, then set $\delta_4 = T_\gamma^{\pm 1}(\delta_1)$).
This observation will be used in the next section.




\section{Spheres}
\label{S:spheres}

In this section we prove Theorem \ref{T:main} for  $S=S_{0,n}$.  If $n \leq 3$, then $\C(S) $ is empty and the result is trivially true.  If $n = 4$, $\C(S)$ is isomorphic to the Farey complex, and thus we may pick any triangle in $\C(S)$ for the subcomplex $\XX$.
We therefore assume $n \geq 5$ and define $\XX \subset \C(S)$ as follows.   
We represent $S$ as the double of a regular $n$--gon $\Delta$ with vertices removed.  An arc connecting non-adjacent sides of $\Delta$ doubles to a curve on $S$.  We let $\XX \subset \C(S)$ denote the subset of curves on $S$ obtained by connecting every non-adjacent pair of sides by a straight line segment and then doubling; see Figure \ref{F:octagonarcs} for the case $n = 8$.  
We index the sides of $\Delta$ in a cyclic order, with indices $1,\ldots,n$, and denote the curve in $\XX$ defined by an arc connecting the $i^{th}$ and $j^{th}$ sides by $\alpha_{i,j}$.

\begin{figure}[htb]
\begin{center}
\includegraphics[width=1.5in,height=1.5in]{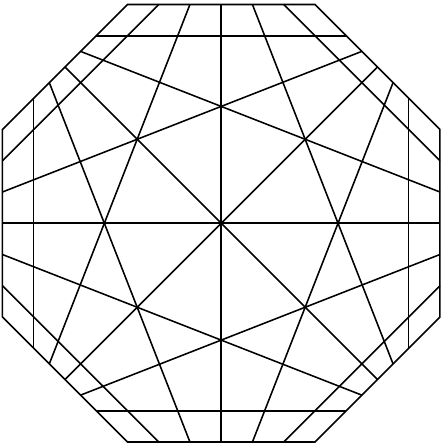} \caption{Octagon and arcs for $S_{0,8}$.} \label{F:octagonarcs}
\end{center}
\end{figure}

We will prove:

\begin{theorem}\label{T:sphere}
For any locally injective simplicial map $\phi:\XX \to \C(S)$, there exists a unique $h \in \Mod(S)$ such that $h|_\XX = \phi$.
\end{theorem}
We note that the pointwise stabilizer $H_\XX < \Mod^\pm(S)$ has order two, and is generated by an orientation-reversing involution (interchanging the two copies of $\Delta$). This is the reason we can conclude the existence of a unique element of $\Mod(S)$, rather than an element of $\Mod^\pm(S)$.

Before starting the proof of Theorem \ref{T:sphere}, we need:

\begin{lemma}
The intersection of any two elements of $\XX$ is $\XX$-detectable.
\label{L:detectsphere}
\end{lemma}

\begin{proof}
The intersection of any two curves $\alpha,\beta \in \XX$ comes from an intersection of the defining arcs $\hat \alpha,\hat \beta$ in $\Delta$.  Take a maximal collection of pairwise disjoint arcs $\hat P$ in $\Delta$ disjoint from $\hat \alpha \cup \hat \beta$.  Then $\hat P \cup \hat \alpha$ and $\hat P \cup \hat \beta$ define the desired pants decompositions $P_\alpha$ and $P_\beta$; see Figure \ref{F:octagondetect}.
\end{proof}

\begin{figure}[htb]
\labellist
\small\hair 2pt
 \pinlabel {$\hat \alpha$} [ ] at 42 78
 \pinlabel {$\hat \beta$} [ ] at 40 25
\endlabellist
\begin{center}
\includegraphics[width=1.5in,height=1.5in]{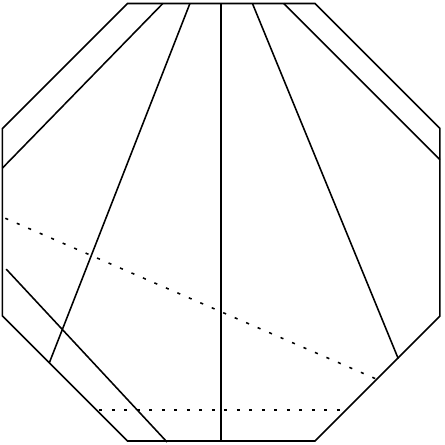}
\caption{Intersecting arcs $\hat \alpha$ and $\hat \beta$ in an octagon, together with $\hat P$, a maximal collection of pairwise disjoint arcs disjoint from $\hat \alpha \cup \hat \beta$.  The dotted lines correspond to curves $\gamma,\delta$ from the proof of Theorem \ref{T:sphere}.}
\label{F:octagondetect}
\end{center}
\end{figure}

\begin{proof}[Proof of Theorem \ref{T:sphere}]

Let $\alpha,\beta\in\XX$ be two curves with $i(\alpha,\beta)\ne 0$.  By inspection there exists $\gamma,\delta \in \XX$ with $i(\alpha,\gamma) = i(\beta,\delta) = i(\gamma,\delta) = 0$ and $i(\gamma,\beta) \neq 0 \neq i(\delta,\alpha)$.  Furthermore, if $P_\alpha,P_\beta$ are the pants decompositions illustrating the $\XX$--detectability of $i(\alpha,\beta) \neq 0$, then we can choose $\gamma,\delta$ to both have nonzero intersection number with exactly one other curve of $P_\alpha - \alpha = P_\beta - \beta$; see Figure \ref{F:octagondetect}.  This curve is of course one of the boundary components of the four-holed sphere filled by $\alpha$ and $\beta$.

By Lemmas \ref{L:minimallyfilling} and \ref{L:detectsphere}, $\phi(\alpha)$ and $\phi(\beta)$ are Farey neighbors or nearly Farey neighbors.  As every curve in $S$ is separating, $\phi(\alpha)$ and $\phi(\beta)$ fill a four-holed sphere $N  = N(\phi(\alpha) \cup \phi(\beta))$.  Because $S$ is a sphere with holes, no two holes of $N$ correspond to the same closed curve in $S$ (such a curve would have to be nonseparating).  By our choice of $\gamma,\delta$, the arcs of intersection of $\phi(\gamma)$ and $\phi(\delta)$ with $N$ have their endpoints on a single boundary component, namely, the $\phi$--image of the unique curve in $P_\alpha - \alpha = P_\beta - \beta$ intersected by  $\phi(\gamma)$ and $\phi(\delta)$.  It follows from the remarks following Lemma \ref{L:detectsphere} that $\phi(\alpha),\phi(\beta)$ are Farey neighbors.

We let $\CC \subset \XX$ denote the set of curves in $\XX$ that bound a disk containing exactly two punctures of $S$.  Equivalently,
\[ \CC = \{ \alpha_{i,j} \in \XX \mid i - j = \pm 2 \mbox{ mod } n \}. \]
We will refer to the curves in $\CC$ as {\em chain} curves  as, taken together, they form a kind of ``cyclic chain'' of curves around $S$; see Figure \ref{F:octagonchain}.
\begin{figure}[htb]
\begin{center}
\includegraphics[width=1.5in, height=1.5in]{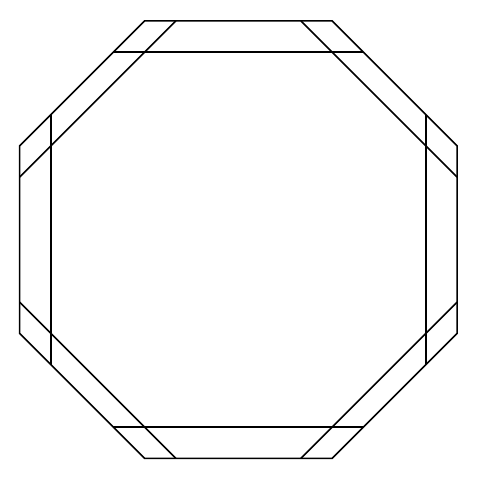} \caption{Arcs defining a chain in an octagon.} \label{F:octagonchain}
\end{center}
\end{figure}

We will refer to a set of curves in $\CC$ as being {\em consecutive} if the indices form an interval in $\{1,\ldots,n\}$ with the {\em cyclic} order obtained by reducing modulo $n$.  Equivalently, a set of curves in $\CC$ is consecutive if their union on $S$ is connected.

Using that $S$ is a (punctured) sphere and so every curve is separating, plus the fact that a pair of consecutive curves are Farey neighbors in a four-holed sphere, we deduce that the neighborhood of any three consecutive curves $\alpha,\beta,\gamma$ in $\CC$ has a regular neighborhood homeomorphic to a regular neighborhood of $\phi(\alpha),\phi(\beta),\phi(\gamma)$.   Since consecutive curves have connected union, this in turn implies that a regular neighborhood $N$ of the union of all curves in $\CC$ is homeomorphic to a regular neighborhood $N'$ of all curves in $\phi(\CC)$ sending each $\alpha_{i,j} \in \CC$ to $\phi(\alpha_{i,j}) \in \phi(\CC)$.

There are $n$ bigons in the boundary of both $N$ and $N'$.  Bigons in $\partial N$ each bound a once-punctured disk.  Since all intersections are essential, it follows that the bigons in $\partial N'$ also bound once-punctured disks.   Since all other boundary components of both $N$ and $N'$ bound (unpunctured) disks, the homeomorphism $N \to N'$ extends to a homeomorphism $h:S \to S$.  Composing with an appropriate element of $H_\XX$ we assume $h$ is orientation-preserving.

Finally, given $ \alpha_{i,j} \in \XX\setminus \CC$, consider the set $A_{i,j} \subset \CC $ of chain curves which are disjoint from $\alpha_{i,j}$; more concretely, $A_{i,j} = \CC \setminus \{\alpha_{i-1,i+1}, \alpha_{j-1,j+1}\}$. Observe that $A_{i,j}$ is an almost filling set, and that $\alpha_{i,j}$ is uniquely determined by $A_{i,j}$.  Since $h$ is a homeomorphism, $h(A_{i,j}) = \phi(A_{i,j})$ is an almost filling set and $h(\alpha_{i,j})$ is uniquely determined by $\phi(A_{i,j})$.  Since $\phi(\alpha_{i,j})$ is also disjoint from $\phi(A_{i,j})$, we must have $\phi(\alpha_{i,j}) = h(\alpha_{i,j})$.  Therefore, $\phi(\alpha) = h(\alpha)$ for all $\alpha \in \XX$ and so Theorem \ref{T:sphere} follows.
\end{proof}

\subsection{Spheres in curve complexes of spheres}
\label{S:spheresofcurves}

Let $\Delta$ be a polygon with $n\ge 4$ vertices. Lee \cite{Lee} proved that the simplicial complex whose $k$-simplices correspond to sets of $k+1$ pairwise disjoint diagonals of $\Delta$
is a simplicial sphere of dimension $n-4$. There is a natural bijection between the set of diagonals of $\Delta$ and the set $\XX$ for $S_{0,n}$, by considering the dual polygon of $\Delta$. Therefore, as mentioned in the introduction, we have:

\begin{theorem}[\cite{Lee}]
The complex $\XX \subset \C(S_{0,n})$ is a simplicial sphere of dimension $n-4$.
\label{T:Lee}
\end{theorem}

A  result of Harer \cite{Harer} states that $\C(S_{0,n})$  is homotopy equivalent to a wedge of spheres of dimension $n-4$. 
We remark that the sphere $\XX\subset \C(S_{0,6})$ had previously been identified by Broaddus \cite{Broaddus}, who proved that  $\XX$ is non-trivial  in
$H_2(\C(S_{0,6}), \mathbb{Z})$ (and hence also in $\pi_2(\C(S_{0,6}))$). Thus, a natural question is to decide whether $\XX$ represents a non-trivial element of $H_{n-4}(\C(S_{0,n}), \mathbb{Z})$  for all $n \ge 7$; compare with Question \ref{Q:homology} in the introduction. 




\section{Tori}
\label{S:tori}

In this section we  prove Theorem \ref{T:main} for $S=S_{1,n}$. If $n\le 1$ then $\C(S)$ is isomorphic to the Farey complex, and again we may choose any triangle in $\C(S)$ for  $\XX$. If $n=2$, the result follows from the previous section, since $\C(S_{1,2})$ and $\C(S_{0,5})$ are isomorphic, although we must replace $\Mod^\pm$ in the conclusion of Theorem \ref{T:sphere} with $\Aut(\C)$; see \cite{Luo}, for instance. 

Thus, from now on we assume that $S$ has $n \geq 3$ punctures. We realize $S$ as the unit square in $\mathbb{R}^2$ with opposite sides identified, with all the punctures of $S$ arranged along a horizontal line of the square. 
Let $\alpha_1, \ldots, \alpha_n$ be curves on $S$, each defined by a vertical segment on the square, and indexed so that $\alpha_i$ and  $\alpha_{i+1}$ together bound an
 annulus containing exactly one puncture of $S$,  for all $i,i+1$ mod $n$. Let $\beta$ be a curve on $S$, defined by a horizontal segment on the square, such that $i(\alpha_j,\beta)=1$ for all $j=1,\ldots, n$; in particular, $\beta$ and $\alpha_j$ are Farey neighbors for all $j$. We set
$$\CC = \{ \alpha_i | i= 1, \ldots, n\} \cup  \{ \beta\};$$ in analogy with Section \ref{S:spheres},  we will refer to the elements of $\CC$ as {\em chain} curves.

Let $\OO$ be the set whose elements are boundary curves of essential subsurfaces filled by connected unions of elements of $\CC$; we will call the elements of $\OO$ {\em outer curves}. See Figure \ref{F:toruscurves} for an example of each type of curve.  We note that each outer curve bounds a disk containing at least two punctures of $S$.
 
 \begin{figure}[htb]
\begin{center}
\includegraphics[width=1.5in, height=1.5in]{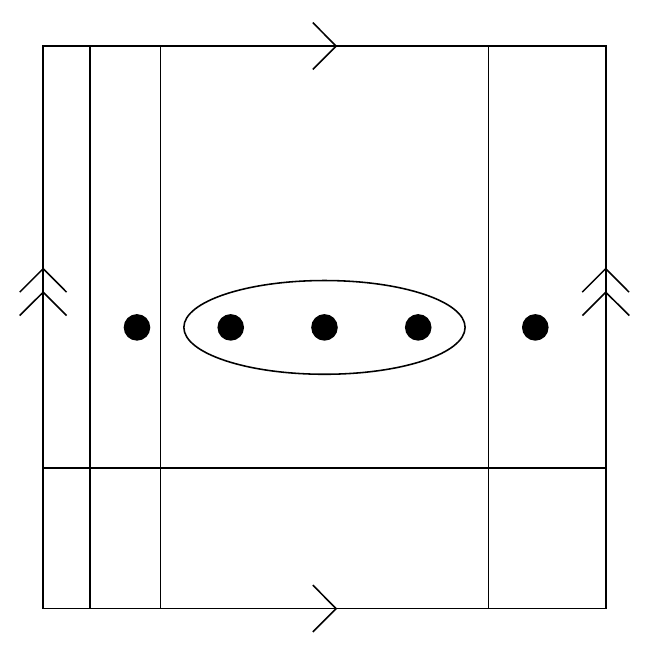} \caption{Four chain curves and an outer curve in $S_{1,5}$.} \label{F:toruscurves}
\end{center}
\end{figure}

 Let $\XX = \CC \cup \OO$. We will prove:

\begin{theorem}
For any locally injective simplicial map $\phi:\XX \to \C(S)$, there exists a unique $h \in \Mod(S)$ such that $h|_\XX = \phi$.
\label{T:tori}
\end{theorem}
As in Section \ref{S:spheres}, $H_\XX < \Mod^\pm(S)$ is generated by an orientation-reversing involution, hence the conclusion again that $h \in \Mod(S)$, versus $\Mod^\pm(S)$.

\medskip

Our first step is to prove that enough intersections are detectable:

\begin{lemma} \label{L:torusdetection}
For any two curves $\gamma_1,\gamma_2 \in \XX$ with $0 < i(\gamma_1,\gamma_2) \leq 2$, their intersection is $\XX$-detectable. 
\label{L:toridetect}
\end{lemma}

\begin{proof} We construct a pants decomposition $P$ of $S \setminus N(\gamma_1\cup \gamma_2)$ as follows; the reader should keep Figure \ref{F:torusdetect} in mind. 
If $\gamma_1,\gamma_2\in \CC$ we take a maximal collection of pairwise disjoint outer curves, each disjoint from $\gamma_1\cup\gamma_2$. If  $\gamma_1\in \OO$ we take a maximal collection of chain curves, each disjoint from $\gamma_1\cup \gamma_2$, plus a maximal collection of outer curves, each disjoint from $\gamma_1\cup\gamma_2$ and the chain curves just constructed. 

Now, we set $P_{\gamma_1}=P\cup \gamma_1$ and $P_{\gamma_2}=P\cup \gamma_2$, observing that $P_{\gamma_1}$ and $P_{\gamma_2}$ are pants decompositions of $S$, and that $P_{\gamma_1}, P_{\gamma_2} \subset \XX$.
\end{proof}

 \begin{figure}[htb]
\begin{center}
\includegraphics[width=4in, height=1.5in]{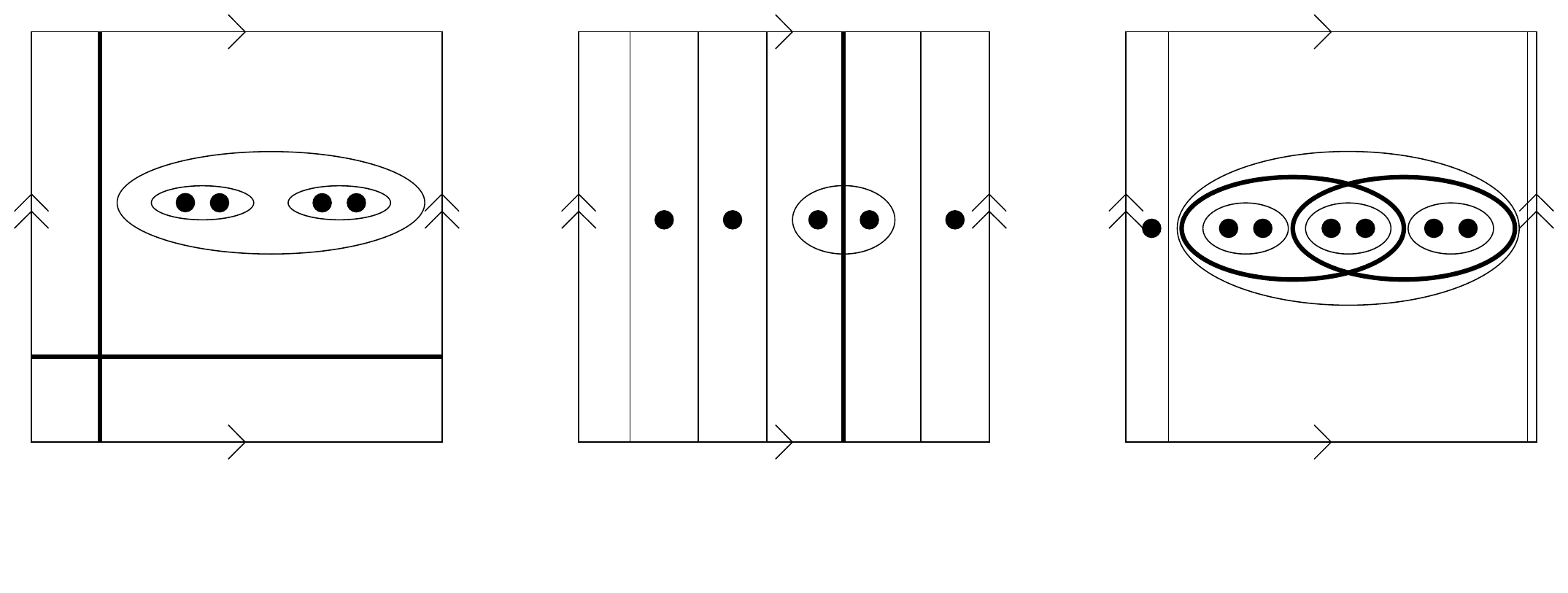} \caption{$\XX$-detectable intersections. Left: The case of two intersecting chain curves (in bold) on $S_{1,4}$. Middle: The case of a chain curve and an outer curve on $S_{1,5}$. Right: The case of two outer curves in $S_{1,7}$.} \label{F:torusdetect}
\end{center}
\end{figure}

\begin{proof}[Proof of Theorem \ref{T:tori}] As a first step, we claim:

\medskip

\noindent{\bf Claim.} For any $i=1, \ldots, n$, we have $i(\phi(\alpha_i),\phi(\beta))=1$. 

\smallskip 

\begin{proof} Let $\alpha_i, \beta \in \CC$, so that $i(\alpha_i, \beta)=1$ and $N_i = N(\alpha_i\cup\beta)$ is a one-holed torus.  Since $n\ge 3$, there is some other chain curve $\alpha_k \in \CC$ and an outer curve $\gamma \in \OO$ so that we can apply Lemma \ref{L:minimallyfilling} to $\gamma,\alpha_i,\beta,\alpha_k$, to conclude that $\phi(\alpha_i)$ and $\phi(\beta)$ are Farey neighbors or nearly Farey neighbors in $N(\phi(\alpha_i)\cup\phi(\beta))$, which we denote simply as $\phi(N_i)$.  It suffices to show that $\phi(N_i)$ is also a one-holed torus.

Let $P \subset \XX$ be the pants decomposition of $S \setminus N_i$ used in the proof of Lemma \ref{L:torusdetection}, which by construction consists entirely of outer curves (compare with Figure \ref{F:torusdetect}), and let $\gamma_0 = \partial N_i \in P$.  For every $\gamma \in P \setminus \gamma_0$, there is $\gamma' \in \OO$ such that $i(\gamma, \gamma') = 2$ and $i(\gamma',\alpha)=i(\gamma',\beta)=0$. By Lemma \ref{L:detectable2detectable}, $i(\phi(\gamma),\phi(\gamma'))\ne 0$, and since $\phi$ is simplicial $i(\phi(\alpha_i),\phi(\gamma') = i(\phi(\beta),\phi(\gamma')) = 0 = i(\phi(\gamma_0),\phi(\gamma'))$.  It follows that $\phi(N_i)$ has a unique boundary component $\phi(\gamma_0)$, and is therefore either a one-holed torus, or else a  sphere with four holes, three of which are punctures of $S$ with the fourth corresponding to $\phi(\gamma_0)$.

Seeking a contradiction, suppose that we are in the latter case.  As  $\phi(\alpha_i)$ and $\phi(\beta)$ are Farey neighbors or nearly Farey neighbors, they each bound a disk containing exactly two punctures of $S$; denote such disks by $D_i$ and $D_\beta$, respectively.  It follows that $\phi(N_j)$ is a four-holed sphere for {\em every} $j$, since it contains $\phi(\alpha_j)$ and $\phi(\beta)$ and $\phi(\beta)$ is separating.

Now observe that $D_i$ and $D_\beta$ have at lease one puncture in common, and since $n \geq 3$, two of $D_1,D_2,D_3$ must also have a puncture in common.  Without loss of generality, suppose it is $D_1,D_2$.  But then $D_1$ and $D_2$ nontrivially intersect, which means that $\phi(\alpha_1) \neq \phi(\alpha_2)$ must also nontrivially intersect.  This contradicts the fact that $\phi$ is simplicial and $i(\alpha_1,\alpha_2) = 0$.  Therefore the claim follows. \end{proof}

By the claim above, the curves $\phi(\alpha_1), \ldots, \phi(\alpha_n)$ are distinct, pairwise disjoint non-separating curves on $S$. Therefore, since $S$ has exactly $n$ punctures, $\phi(\alpha_i)$ and $\phi(\alpha_j)$ together bound an annulus with punctures for all $i\ne j$. Moreover,  since $n \geq 3$, for all $i=1, \ldots, n$ there exists a (unique) curve $\gamma\in \OO$ such that $i(\alpha_i, \gamma)\ne 0$, $ i(\alpha_{i+1}, \gamma)\ne 0$ and $i(\alpha_k, \gamma)=0$ for all $k\notin\{i, i+1\}$. By Lemma \ref{L:detectable2detectable} and \ref{L:toridetect}, $i(\phi(\alpha_i), \phi(\gamma))\ne 0$, $i(\phi(\alpha_{i+1}), \phi(\gamma))\ne 0$ and $i(\alpha_k, \gamma)=0$ for all $k\notin\{i, i+1\}$. It follows that $\phi(\alpha_i)$ and $\phi(\alpha_{i+1})$ together bound a once-punctured annulus. 
Therefore, as in Section \ref{S:spheres}, we can construct a unique orientation-preserving homeomorphism $h:S\to S$, which satisfies $\phi(\alpha)=h(\alpha)$ for all $\alpha \in \CC$. 

It remains to show that $\phi(\gamma)=h(\gamma)$ for all $\gamma\in \OO$. There is a natural partition $\OO=\bigcup_{k=2}^n\OO_k$ in terms of the number $k$ of punctures of $S$ contained in the disk bounded by the corresponding outer curve.
Every $\gamma \in \OO_2$ is uniquely determined by an almost filling subset $A\subset \CC$; namely, $A$  consists of $\beta$ and the $n-2$ elements of $\{\alpha_i\}_{i=1}^n$ that are disjoint from $\gamma$. Thus $\phi(\gamma) = h(\gamma)$  for all $\gamma\in \OO_2$. 

If $k\ge 3$, note that any element in $\delta \in \OO_k$ is uniquely determined by an almost filling set $B\subset \CC\cup\OO_2 \cup \ldots \cup  \OO_{k-1}$ consisting of all the curves disjoint from $\delta$.
Therefore, by induction on $k$, $\phi(\delta)=h(\delta)$ for all $\delta\in \OO_k$, and hence for every $\delta \in \OO$.   This concludes the proof of Theorem \ref{T:tori} 
\end{proof}

\section{Closed surfaces of higher genus}
\label{S:closed}

In this section we prove Theorem \ref{T:main} for closed surfaces of genus $g\ge 2$.  The proof is similar to the previous cases, but since the possible configurations of curves we need to consider is more complicated (and will become even more so in the next section), the notation becomes more involved.
Before we begin, we first note that if $S=S_{2,0}$, the result follows from Theorem \ref{T:sphere}, since $\C(S_{2,0})$ and $\C(S_{0,6})$ are isomorphic; see \cite{Luo}, for instance.  So we assume $S = S_{g,0}$ with $g \geq 3$; compare with Remark \ref{R:genus2same}.

Let $\CC=\{\alpha_0, \ldots\alpha_{2g+1}\}$ be a set of curves on $S$ such that $i(\alpha_i,\alpha_j)= 1$ if $\vert i-j \vert =1$ mod $2g+2$ and  
$i(\alpha_i,\alpha_j)= 0$ otherwise. See Figure \ref{F:genus4chain}.  We will refer to the elements of $\CC$ as {\em chain curves}, by analogy with the previous cases.  Indeed, taking a branched cover over the sphere with $2g+2$ marked points gives a closed surface of genus $g$, and the preimage of the chain curves on the sphere become chain curves on $S$.  
\begin{figure}[htb]
\begin{center}
\includegraphics[width=2.5in]{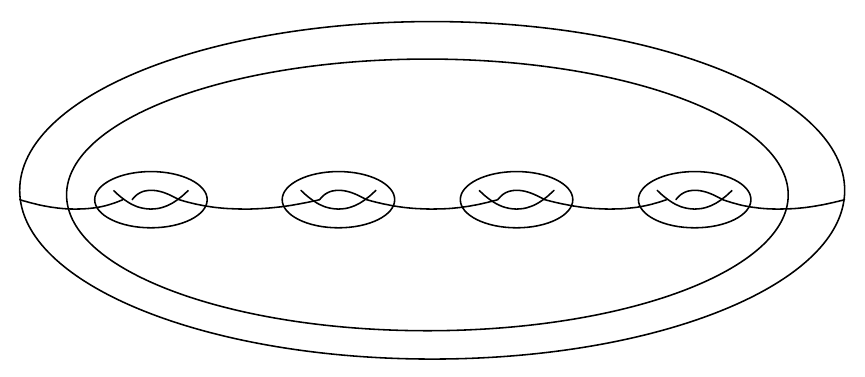} \caption{Chain curves $\CC$ on a genus $4$ surface.} \label{F:genus4chain}
\end{center}
\end{figure}

Given an interval $J \subset \{0,\ldots,2g+1 \}$ with respect to the {\em cyclic} order coming from reducing modulo $2g+2$, we set
\[ \alpha_J = \bigcup_{j \in J} \alpha_j.\]
For  $1 < |J| < 2g$, let $N(\alpha_J)$ denote the subsurface of $S$ filled by $\alpha_J$; equivalently, $N(\alpha_J)$ is the regular neighborhood of $\alpha_J$ in $S$.

When $|J|$ is even, $\partial N(\alpha_J)$ is a separating curve, denoted $\sigma_J$.  Let $\SS$ be the set of all such separating curves.   On the other hand, if $|J| < 2g-1$ is odd, then $\partial N(\alpha_J)$ is a bounding pair, denoted $\beta_J^+ \cup \beta_J^-$;  for $|J| = 2g-1$, $\partial N(\alpha_J)$ consists of two isotopic copies of $\alpha_k$, where $k = k(J)$ is the unique odd integer modulo $2g+2$ not in $J$.

We let $\BB$ denote the set of all such bounding pairs $\beta_J^\pm$ (thus {\em not} including the case $|J| = 2g-1$).
We partition $\BB$ as follows.  First, observe that the union of the even chain curves $\alpha_0 \cup \alpha_2 \cup \ldots \cup \alpha_{2g}$ cuts $S$ into $\Theta_e^+,\Theta_e^-$, each of which is homeomorphic to $S_{0,g+1}$; we will refer to these  the ``top'' and the ``bottom'' spheres.  See Figure \ref{F:topandbottom}.  Similarly, the union of the odd chain curves $\alpha_1 \cup \alpha_3 \cup \ldots \cup \alpha_{2g+1}$ cuts $S$ up into $\Theta_o^+,\Theta_o^-$, each also homeomorphic to $S_{0,g+1}$: the ``front'' and ``back'' spheres.   Each $\gamma \in \BB$ is contained in exactly one of these, and we accordingly partition $\BB = \BB^+_e \cup \BB^-_e \cup \BB^+_o \cup \BB^-_o$.  

\begin{figure}[htb]
\labellist
\small\hair 2pt
 \pinlabel {$\Theta_e^+$} [ ] at 123 79
 \pinlabel {$\Theta_e^-$} [ ] at 123 29
\endlabellist
\centering
\includegraphics[scale=.75]{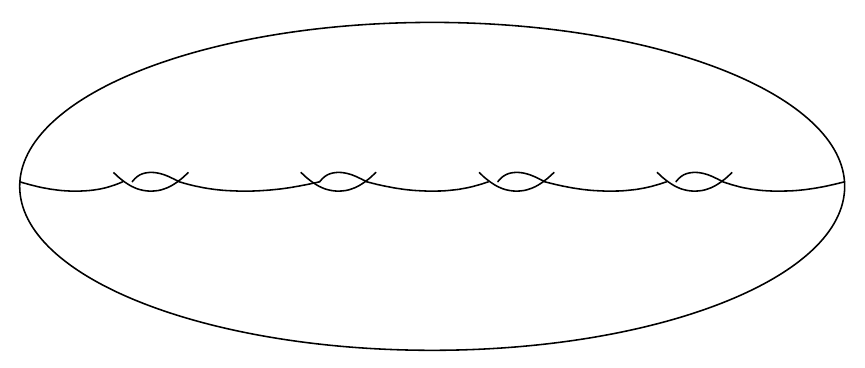}
\caption{The top and bottom spheres-with-$(g+1)$-holes, $\Theta_e^+,\Theta_e^-$.}
\label{F:topandbottom}
\end{figure}

Let $\XX_0 = \CC \cup \SS \cup \BB$. We need one more type of curve before defining our set $\XX$. Observe that for any interval $J$, $\partial N(\alpha_J)$ intersects $\alpha_j$ if and only if $j$ is an immediate predecessor or successor of $J$.  For concreteness, suppose $j$ is an immediate successor, so $J = \{i,i+1,\ldots, j-1 \}$.  When $|J| < 2g-1$ is odd, then
$\partial N(\alpha_J)$ consists of the bounding pair $\beta_J^\pm$, and $i(\alpha_j,\beta_J^+) = i(\alpha_j,\beta_J^-) = 1$. We let $\mu_{j,J}^\pm$ denote the boundary of the one-holed torus $N(\alpha_j \cup \beta_J^\pm)$, and let $\UU$ denote the set of all such curves $\mu_{j,J}^\pm$; compare Figure \ref{F:localclosed}.   We set:
\[ \XX = \XX_0 \cup \UU = \CC \cup \SS \cup \BB \cup \UU.\]

\begin{theorem} \label{T:closedsurfaces} For every locally injective simplicial map $\phi:\XX \to \C(S)$, there exists a unique mapping class $h \in \Mod^\pm(S)$ such that $h|_\XX = \phi$.
\end{theorem}

\begin{remark}
As will become apparent, the set $\UU$ will only be used to detect intersections between elements of $\XX_0$.
\end{remark}

\begin{proof}
We first describe a method for constructing pants decompositions from curves in $\XX_0$.   For any $i$, consider the increasing sequence of intervals $J_1 \subset J_2 \subset J_3 \subset \ldots \subset J_{2g-2}$, with $J_k = \{i,i+1,\ldots,i + k \}$, noting that $N(\alpha_{J_k})$ is obtained from $N(\alpha_{J_{k-1}})$ by adding a single one-handle to $\partial N(\alpha_{J_{k-1}})$.  As such, the union of all boundary curves
\[ \begin{array}{lll} P &  = & \partial N(\alpha_{J_1}) \cup \partial N(\alpha_{J_2}) \cup \cdots \cup \partial N(\alpha_{J_{2g-4}}) \cup \partial N(\alpha_{J_{2g-3}}) \cup \partial N(\alpha_{J_{2g-2}}) \\
 & = & \sigma_{J_1} \cup \beta_{J_2}^- \cup \beta_{J_2}^+ \cup \cdots \cup \beta_{J_{2g-4}}^- \cup \beta_{J_{2g-4}}^+ \cup \sigma_{J_{2g-3}} \cup \alpha_{i-2} \end{array} \]
is a pants decomposition of the complement of the one-holed torus $N(\alpha_{J_1}) = N(\alpha_{i}\cup\alpha_{i+1})$; here, $\partial N(\alpha_{J_{2g-2}})$ is a union of two copies of $\alpha_{i-2}$, and we take just one in $P$.

This immediately implies $i(\alpha_i,\alpha_{i+1}) \neq 0$ is $\XX$--detectable.  Consequently, applying this argument to four consecutive chain curves and appealing to Lemma \ref{L:minimallyfilling}, it follows that $\phi(\alpha_i),\phi(\alpha_{i+1})$ are Farey neighbors or nearly Farey neighbors.

Next, we want to show that $N(\phi(\alpha_i)\cup\phi(\alpha_{i+1}))$ is a one-holed torus, so that we may conclude that $i(\phi(\alpha_i),\phi(\alpha_{i+1})) = 1$.  Let $P$ be the pants decomposition of the complement of $N(\alpha_i\cup\alpha_{i+1})$ containing $\sigma_{J_1} = \partial N(\alpha_i\cup\alpha_{i+1})$ just constructed.
\begin{claim}
For all $\gamma \in P  \setminus \sigma_{J_1}$, there exists $\gamma' \in \XX$ such that $i(\gamma,\gamma') \neq 0$ is $\XX$--detectable and $\gamma'$ is disjoint from $N(\alpha_i\cup\alpha_{i+1})$.
\end{claim}
\begin{proof}  
We first observe that if $\gamma = \sigma_{J_k} \in P \setminus \{\sigma_{J_1} \}$, then $\sigma_{J_k},\alpha_{i+k+1}$ are Farey neighbors in the four-holed sphere $N(\sigma_{J_k}\cup\alpha_{i+k+1})$, with boundary components
\[ \partial N(\sigma_{J_k}\cup\alpha_{i+k+1}) = \beta_{J_{k-1}}^+ \cup \beta_{J_{k-1}}^- \cup \beta_{J_{k+1}}^+ \cup \beta_{J_{k+1}}^- . \]
See Figure \ref{F:localclosed}.  Therefore, 
\[ P_{\sigma_{J_k}} = P  \cup \alpha_i \, \, \mbox{ and } \, \,  P_{\alpha_{i+k+1}} = (P \setminus \sigma_{J_k}) \cup \alpha_{i+k+1} \cup \alpha_i  \]
are two pants decompositions proving that  $i(\sigma_{J_k},\alpha_{i+k+1}) \neq 0$ is $\XX$--detectable.  Since $i(\alpha_i,\alpha_{i+k+1}) = i(\alpha_{i+1},\alpha_{i+k+1}) = 0$, we have proven the claim for $\gamma = \sigma_{J_k}$, by setting $\gamma' = \alpha_{i+k+1}$.

\begin{figure}[htb]
\labellist
\small\hair 2pt
 \pinlabel {$\alpha_{i+k+1}$} [ ] at 97 61
 \pinlabel {$\sigma_{J_k}$} [ ] at 80 20
 \pinlabel {$\beta_{J_{k-1}}^+$} [ ] at 16 62
 \pinlabel {$\beta_{J_{k-1}}^-$} [ ] at 16 16
 \pinlabel {$\beta_{J_{k+1}}^+$} [ ] at 138 62
 \pinlabel {$\beta_{J_{k+1}}^-$} [ ] at 138 16
 \pinlabel {$\sigma_{J_{k-1}}$} [ ] at 222 63
 \pinlabel {$\mu_{i+k+1,J_k}^+$} [ ] at 278 8
 \pinlabel {$\beta_{J_k}^+$} [ ] at 267 68
 \pinlabel {$\alpha_{i+k+1}$} [ ] at 296 60
 \pinlabel {$\sigma_{J_{k+1}}$} [ ] at 343 65
\endlabellist
\centering
\includegraphics[scale=1.0]{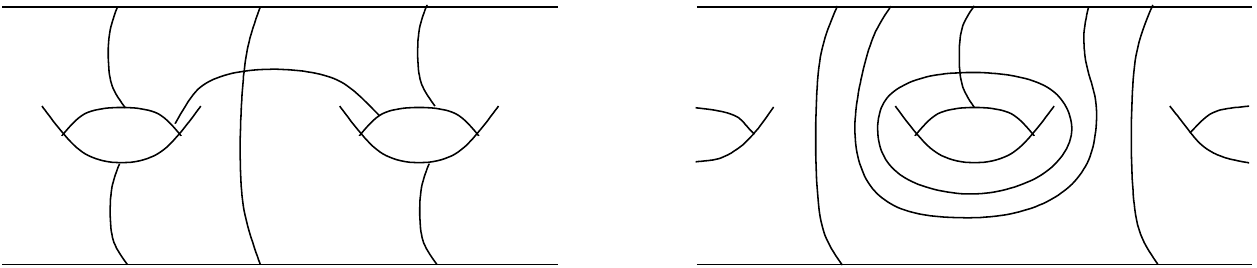}
\caption{$\XX$--detectability of $i(\gamma,\gamma') \neq 0$ for $\gamma' = \alpha_{i + k + 1}$ and $\gamma = \sigma_{J_k}$ (left) and $\gamma = \beta_{J_k}^+$ (right).}
\label{F:localclosed}
\end{figure}
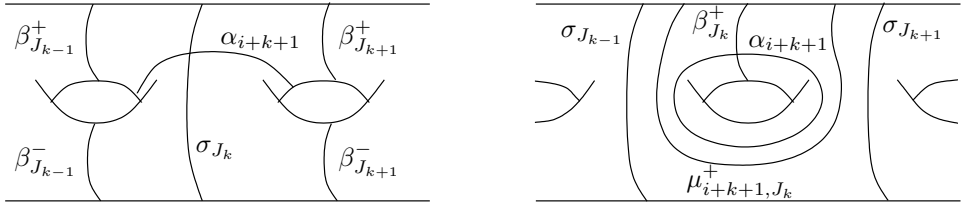

For $\gamma = \beta_{J_k}^+$, we again let $\gamma' = \alpha_{i + k +1 }$ and consider the pants decompositions
\[ P_{\beta_{J_k}^+} =  (P \setminus \beta_{J_k}^- )\cup \mu^+_{i+k+1,J_k} \cup \alpha_i \]
 and 
\[ P_{\alpha_{i+k + 1}}  = P \setminus (\beta_{J_k}^- \cup \beta_{J_k}^+) \cup \mu_{i+k + 1,J_k}^+ \cup \alpha_{i + k + 1} \cup \alpha_i.\]
See Figure \ref{F:localclosed}.
Using these pants decompositions, we see that $i(\beta_{J_k}^+,\alpha_{i + k + 1}) \neq 0$ is $\XX$--detectable, while again $i(\alpha_i,\alpha_{i+k+1}) = i(\alpha_{i+1},\alpha_{i+k+1}) = 0$.  The case of $\gamma = \beta_{J_k}^-$ is proven in exactly the same way as $\beta_{J_k}^+$, and this completes the proof of the claim.
\end{proof}

As in Section \ref{S:tori}, it follows that, for all $i$, $N(\phi(\alpha_i)\cup\phi(\alpha_{i+1}))$ has only one boundary component in $S$. 
Since $S$ has no punctures, $N(\phi(\alpha_i)\cup\phi(\alpha_{i+1}))$ is a one-holed torus and thus $i(\phi(\alpha_i),\phi(\alpha_{i+1})) = 1$.

Since  $S$ has genus $g$, we deduce  that there is a homeomorphism $h_0:S \to S$ satisfying $h_0(\alpha) = \phi(\alpha)$ for all $\alpha \in \CC$. Moreover, $h_0$ is unique up to precomposing with elements in the point-wise stabilizer $H_\CC$ of $\CC$.

Next, each curve $\sigma_J \in \SS$ is uniquely determined by the almost filling set $\alpha_J \cup \alpha_{J'} \subset \CC$, where $J = \{i,\ldots,i+k \}$ and $J' = \{ i + k + 2,\ldots, i-2 \}$ (recall the indices are taken modulo $2g+2$).
Since $\phi(\sigma_J)$ is disjoint from all curves in $\phi(\alpha_J) \cup \phi(\alpha_{J'})$, then $\phi(\sigma_J)$ is also uniquely determined by $\phi(\alpha_J) \cup \phi(\alpha_{J'})$, and thus $\phi(\sigma_J) = h_0(\sigma_J)$, for all $\sigma_J \in \SS_J$. 

An analogous argument {\em almost} works for a bounding pair $\beta_J^\pm \in \BB$.  Here, $\{ \beta_J^+,\beta_J^- \}$ is similarly determined by $\alpha_J \cup \alpha_{J'}$.  As in the previous paragraph, we can conclude that $\{h_0(\beta_J^+),h_0(\beta_J^-) \} = \{ \phi(\beta_J^+),\phi(\beta_J^-) \}$.  However, it is not necessarily the case that $\phi(\beta_J^\pm) =h_0(\beta_J^\pm)$, since $H_\CC$ acts nontrivially on $\BB$.

We therefore wish to precompose $h_0$ with some element $f \in H_\CC$ so that $h_0 \circ f(\beta_J^\pm) = \phi(\beta_J^\pm)$.  To choose the appropriate element $f \in H_\CC$ we proceed as follows.
The $\phi$--images of the even (respectively, odd) chain curves cut $S$ into two complementary surfaces, each homeomorphic to $S_{0,g+1}$, which we denote $\Omega_e^\pm$ (respectively, $\Omega_o^\pm$). These determine partitions $\phi(\BB) = \phi(\BB)_e^+ \cup \phi(\BB)_e^- \cup \phi(\BB)_o^+ \cup \phi(\BB)_o^-$.  Moreover,
\[ \phi(\BB_e^+ \cup \BB_e^-) = \phi(\BB)_e^+ \cup \phi(\BB)_e^- \mbox{ and } \phi(\BB_o^+ \cup \BB_o^-) = \phi(\BB)_o^+ \cup \phi(\BB)_o^- \]
since, for example, being in $\BB_e^+ \cup \BB_e^-$ is determined by disjointness with the even chain curves, and $\phi$ is simplicial.
 
By inspection, any two curves in the same subset $\BB_e^+$, $\BB_e^-$, $\BB_o^+$, or $\BB_o^-$ are disjoint, or else there is a third curve in the same subset having nonzero intersection with both -- equivalently, the union of the curves in any subset of the partition forms a connected subset of $S$.  Now, the intersection between any two curves in $\BB_e^+$, say, is $\XX$--detectable; indeed, the subcomplex $\BB_e^+$ is isomorphic to the finite rigid set $\XX$  for $S_{0,g+1}$ constructed in Section \ref{S:spheres}, and similarly for the other four subsets of the partition of $\BB$. Therefore, $\phi$ sends each subset of the partition of $\BB$ to one of the subsets of the partition of $\phi(\BB)$.  

The group $H_\CC$ is generated by an involution $i_e:S \to S$ interchanging  $\Theta_e^+$ and $\Theta_e^-$, and an involution $i_o:S \to S$ interchanging $\Theta_o^+$ and $\Theta_o^-$, with $H_\CC \cong \mathbb Z/2 \mathbb Z \times \mathbb Z/2 \mathbb Z$.   So, precomposing $h_0$ with an appropriate element $f \in H_\CC$, we can assume $h = h_0 \circ f$ satisfies $h(\beta_J^\pm) = \phi(\beta_J^\pm)$ for all $\beta_J^\pm \in \BB$.

Now we know $h(\gamma) = \phi(\gamma)$ for all $\gamma \in \XX_0$.  Finally any $\mu_{j,J}^\pm \in \UU$  is uniquely determined by an almost filling subset of $\XX_0$, namely  $\beta_J^\pm \cup \alpha_j \cup P'$, where $P'$ is the pants decomposition of the complement of $N(\beta_J^\pm\cup\alpha_j)$.
Therefore, $h(\mu_{j,J}^\pm) = \phi(\mu_{j,J}^\pm)$ and $h|_\XX = \phi$ as required.
\end{proof}

\begin{remark} \label{R:genus2same}
The proof above is also valid for $S_{2,0}$.  However, in that case $\BB = \emptyset$ and $\XX = \CC \cup \SS$, so that $H_\CC = H_\XX$, and the proof ends after the construction of $h_0:S \to S$ and verification that $h_0(\sigma_J) = \phi(\sigma_J)$ for $\alpha \in \SS$.
\end{remark}

\section{Higher genus punctured surfaces.}

 Finally, we prove Theorem \ref{T:main} for $S$ a surface of genus $g\ge 2$ and with $n\ge 1$ punctures. Let $\alpha_0^0,\ldots,\alpha_0^n,\alpha_1,\ldots,\alpha_{2g+1}$ be curves on $S$ indexed as follows (see Figure \ref{F:puncturedex}). For all $i=0, \ldots, n-1$, the curves $\alpha_0^i,\alpha_0^{i+1}$ together bound a once-punctured annulus.
For all $1 \leq i,j \leq 2g+1$ we have \[ i(\alpha_i,\alpha_j) = \left\{ \begin{array}{ll} 1 & \mbox{ if } |i-j| = 1\\ 0 & \mbox{ otherwise; } \end{array} \right. \] Finally, for $1 \leq i \leq 2g+1$ and $0 \leq j \leq n$ we have
\[ i(\alpha_i,\alpha_0^j) = \left\{ \begin{array}{ll} 1 & \mbox{ if } i = 1 \mbox{ or } 2g+1 \\ 0 & \mbox{ otherwise. } \end{array} \right. \]
We set $\CC = \{\alpha_0^0,\ldots,\alpha_0^n,\alpha_1,\ldots,\alpha_{2g+1}\}$ and
$\CC_f = \{\alpha_0^0, \ldots,\alpha_0^n \}\subset \CC$.

\begin{figure}[htb]
\labellist
\small\hair 2pt
 \pinlabel {$\alpha_{2g+1}$} [ ] at 340 20
 \pinlabel {$\alpha_1$} [ ] at 96 -2
 \pinlabel {$\alpha_2$} [ ] at 114 5
 \pinlabel {$\alpha_3$} [ ] at 141 -1
 \pinlabel {$\alpha_0^4$} [ ] at 46 24
 \pinlabel {$\alpha_0^3$} [ ] at 12 43
 \pinlabel {$\alpha_0^2$} [ ] at -5 73
 \pinlabel {$\alpha_0^1$} [ ] at 15 120
 \pinlabel {$\alpha_0^0$} [ ] at 48 139
\endlabellist
\centering
\includegraphics[scale=.6]{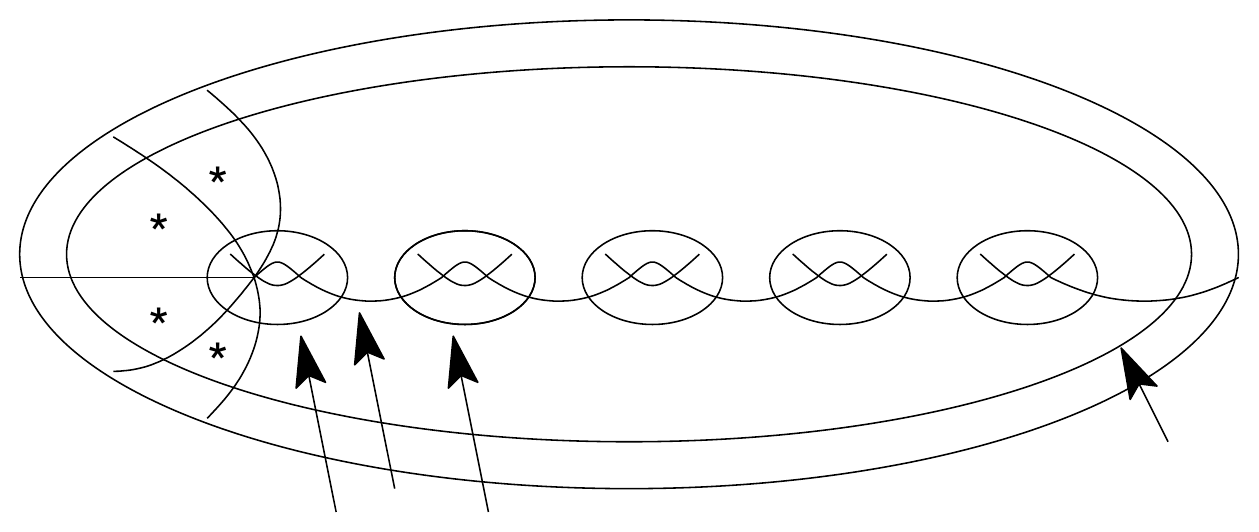}
\caption{ The set $\CC$ for $S_{5,4}$. }
\label{F:puncturedex}
\end{figure}

We will construct the finite rigid set $\XX$ using these curves as building blocks.  The proof that the resulting set is rigid will mimic both the proof for punctured tori and for closed surfaces.  We therefore define the rest of the curves in $\XX$ in a way that highlights this analogy.

\subsection{Torus comparison}

For each $0 \leq i \leq j \leq n$, consider the following essential subsurfaces:
\[ N_1^{i,j} = N(\alpha_0^i \cup \alpha_0^j \cup \alpha_1) \, , \quad N_2^{i,j} = N(\alpha_0^i \cup \alpha_0^j \cup \alpha_1 \cup \alpha_2) \]
\[ N_{2g+1}^{i,j} = N(\alpha_0^i \cup \alpha_0^j \cup \alpha_{2g+1}) \, , \quad N_{2g}^{i,j} = N(\alpha_0^i \cup \alpha_0^j \cup \alpha_{2g+1} \cup \alpha_{2g}) \]

The surface $N_1^{i,j}$ is two-holed torus if $j-i \geq 1$. If $j-i = 1$, then one of the holes is a puncture of $S$, while if $j-i > 1$, it is a boundary component which we denote $\epsilon^{i,j}_1$, which itself bounds a disk with $j-i$ punctures of $S$.
The other boundary component of $N_1^{i,j}$, denoted $\sigma^{i,j}_1$, bounds a torus with $j-i+1$ holes, where exactly $j-i$
of them correspond to punctures of $S$. Similarly, we let $\epsilon^{i,j}_{2g+1}$ and $\sigma^{i,j}_{2g+1}$ denote the boundary components of $N_{2g+1}^{i,j}$ (the former when $j-i > 1$), labeled in an analogous manner.  Finally, if $i = j$  then $N_1^{i,j}$ and  $N_{2g+1}^{i,j}$ are one-holed tori, whose boundary components we call $\sigma^{i,i}_1$ and $\sigma^{i,i}_{2g+1}$, respectively.

Observe that $\epsilon^{i,j}_1 = \epsilon^{i,j}_{2g+1}$ for all $i$ and $j$, and so we simply denote this $\epsilon^{i,j}$.
Define
 \[\OO = \{\epsilon^{i,j} \mid j-i \geq 2 \}.\]
The elements of $\OO$ are called {\em outer curves}, by way of comparison with Section \ref{S:tori}; again, there is a natural partition $\OO=\bigcup_{k=2}^n \OO_k$ in terms of  $k=j-i$.  We also define
\[ \SS_T = \{ \sigma^{i,j}_\ell \mid \ell = 1 \mbox{ or } 2g+1 \, , \, 0 \leq i \leq j \leq n \}.\]

\begin{remark} \label{R:comparetorus}
$S\setminus \sigma_1^{0,n}$ has two connected components, which we denote as  $S_L$ and $S_R$, where $S_L$ has genus 1 and contains all the punctures of $S$,  and $S_R$ has genus $g-1$. Observe that for $n \geq 2$ the set $\{\gamma\in \CC \cup \OO \cup \SS_T  \mid \gamma\subset S_L\}$ is exactly the finite rigid set for $S_L=S_{1,n+1}$ identified in Section \ref{S:tori}.  The chain curves from Section \ref{S:tori} correspond to elements of $\CC$ contained in $S_L$, while the outer curves from Section \ref{S:tori} correspond to curves in $\OO$ {\em and} $\SS_T$.
\end{remark}

Now, the surface $N_2^{i,j}$ is a three-holed torus (when $j-i > 0$), with one boundary component being $\epsilon^{i,j}$ (when $j - i > 1$). We denote the other two components $\beta_{0,1,2}^{i,j,+}$ and $\beta_{0,1,2}^{i,j,-}$, respectively, with the convention that in Figure \ref{F:puncturedex} $\beta_{0,1,2}^{i,j,+}$ would be on the ``top'' and $\beta_{0,1,2}^{i,j,-}$ would be on the ``bottom''.  This convention can also be described in terms of various intersection numbers (the case $g =2$ and $i=0$ and $j=n$ these are the same curve in $S$, namely $\alpha_4$, but we ignore this).
In fact, $\beta_{0,1,2}^{i,j,+}$ depends only on $i$, and not $j$, so we denote it $\beta_{0,1,2}^{i,+}$ when convenient.  Similarly, we write $\beta_{0,1,2}^{j,-}$ for $\beta_{0,1,2}^{i,j,-}$.

In an analogous fashion, we define $\beta_{2g,2g+1,0}^{i,+}$ and $\beta_{2g,2g+1,0}^{j,-}$ from $N_{2g}^{i,j}$.  We now set
\[ \BB_T^\pm = \{ \beta_J^{i,\pm} \mid J = \{0,1,2 \} \mbox{ or } \{2g, 2g+1,0 \} \, , \, 1 \leq i  \leq n \} \]
and $\BB_T = \BB_T^+ \cup \BB_T^-$.  These curves do not have obvious analogues in the case of punctured tori, but they will be used as a sort of ``bridge'' to the rest of the surface. 

\subsection{Closed surface comparison}

We define $\CC_0 = \{ \alpha_0^1,\alpha_1,\alpha_2,\ldots,\alpha_{2g+1} \}$; these curves will play much the same role as the chain curves in the closed case.  We write $\alpha_0 = \alpha_0^1$ to simplify the notation.  Let $J \subset  \{ 0,\ldots, 2g+1 \}$ be a proper interval in the cyclic order modulo $2g+2$ with $|J| \leq 2g$, and let $\alpha_J = \bigcup_{j \in J}\alpha_j$.

If $|J|$ is even, we write $\sigma_J = \partial N(\alpha_J) $, and define $\SS_0$ to be the set of all such separating curves $\sigma_J$. 
We note that $\SS_0 \cap \SS_T\neq \emptyset$, as $\sigma_{\{0,1\}} = \sigma^{1,1}_1$ and $\sigma_{\{2g+1,0\}} = \sigma^{1,1}_{2g+1}$.  When $n = 1$, we do not consider the case $|J| = 2g$ as then $\alpha_J$ fills $S$ (so $\alpha_J$ would be empty).

Next, let $\BB_0$ be the set of bounding pairs which occur as a boundary component of  subsurfaces $N(\alpha_J)$, where $|J| < 2g$ is odd and $J$ starts and ends with an even number.  As in the closed case, the even curves of $\CC_0$ divide $S$ into the {\em top sphere}, denoted $\Theta_e^+$, and the {\em bottom sphere}, denoted $\Theta_e^-$.  However, in this case $\Theta_e^+$ has $g+2$ holes, one of which is a puncture of $S$,  whereas $\Theta_e^-$ has $g+n$ holes, $n-1$ of which are punctures of $S$.  For each such $J$, write $\beta_J^\pm$ for the two boundary components of $N(\alpha_J)$, with $\beta_J^+ \subset \Theta_e^+$ and $\beta_J^- \subset \Theta_e^-$.  The set $\BB_0$ is then naturally partitioned as $\BB_0 = \BB_0^+ \cup \BB_0^-$; again, we note that $\BB_0 \cap \BB_T \neq \emptyset$.  Note that unlike the closed case, even for $|J|= 2g-1$ we get two distinct curves $\beta_J^\pm$ (which cobound an $n$--punctured annulus).

Also as in the closed case, for $J$ as in the previous paragraph, we consider the immediate successor or predecessor  $j$ to $J$ (so that $\{j \} \cup J$ is an interval with $| \{ j \} \cup J | = |J| + 1$).   Then $i(\alpha_j,\beta_J^+) = i(\alpha_j,\beta_J^-) = 1$, and we set $\mu_{j,J}^\pm$ to denote the boundary of the subsurface filled by $\alpha_j$ and $\beta_J^\pm$.   The union of the $\mu_{j,J}^\pm$ will be denoted $\UU$.

Now define 
\[ \XX = \CC \cup \OO \cup \SS_T \cup \SS_0 \cup \BB_T \cup \BB_0 \cup \UU. \]

\begin{theorem}
\label{T:punctured}
For any locally injective simplicial map $\phi:\XX \to \C(S)$, there exists a unique $h \in \Mod^\pm(S)$ such that $\phi = h|_\XX$.
\end{theorem}

\begin{proof}
We start by proving that for all $\alpha,\alpha' \in \CC$, with $i(\alpha,\alpha') = 1$ we have $i(\phi(\alpha), \phi(\alpha')) = 1$.   The first step is to prove:

\begin{claim}
Let  $\alpha,\alpha' \in \CC$ with $i(\alpha,\alpha') = 1$. Then the intersection of $\alpha$ and $\alpha'$ is $\XX$-detectable. 
\end{claim}

\begin{proof} [Proof of claim.]  Before starting, note that for any $\epsilon^{i,j}\in \OO$, the multicurve $\epsilon^{i+1,j} \cup \ldots \cup \epsilon^{j-2,j}$ is a pants decomposition of the punctured disk bounded by $\epsilon^{i,j}$. 

Now, there are two cases to consider.

\bigskip

\noindent
{\bf Case 1.} Neither $\alpha$ nor $\alpha'$ is in $\CC_f$.\smallskip

\noindent

Let $\alpha= \alpha_j$ and $\alpha'=\alpha_{j+1}$, and first assume that $j$ is even (and thus $j \neq 0$).   Then let
$ J_1 \subset J_2 \subset \ldots \subset J_{2g}$
be an increasing sequence of intervals, with $J_k = \{j,j+1, \ldots, j+k \}$ and indices taken modulo $2g+2$.
For $n \geq 3$, this determines a pants decomposition of $S \setminus N(\alpha_j \cup \alpha_{j+1})$:
\[ P = \sigma_{J_1} \cup \beta^-_{J_2} \cup \beta^+_{J_2} \cup \sigma_{J_3} \cup \ldots \cup \beta_{J_{2g-2}}^- \cup \beta_{2g-2}^+ \cup \sigma_{J_{2g-1}} \cup \epsilon^{1,n} \cup \ldots \cup \epsilon^{n-2,n}. \]
This is similar to Section \ref{S:closed}, except we must add outer curves at the end.  For $n \leq 2$, we have fewer curves.  Specifically, for $n = 2$, there are no outer curves so $P$ ends with $\sigma_{J_{2g-1}}$.  For $n = 1$, $P$ ends with $\beta_{J_{2g-2}}^- \cup \beta_{J_{2g-2}}^+$.  For simplicity of notation, we ignore these cases as the arguments are actually simpler there.  Setting $P_\alpha= P\cup \alpha$ and $P_{\alpha'}=P\cup \alpha'$, it follows that the intersection between $\alpha,\alpha'$ is $\XX$--detectable.

When $j$ is odd, we can apply the same argument, this time with the intervals $J_k$  having the form $J_k = \{j+1,j,j-1,\ldots,j-k+1 \}$. This completes the proof of Case 1.

\bigskip

\noindent
{\bf Case 2.} One of $\alpha$ or $\alpha'$ is in $\CC_f$.\smallskip

\noindent
Without loss of generality, assume $\alpha = \alpha_0^i$ and $\alpha' = \alpha_1$, as the case $\alpha' = \alpha_{2g+1}$ is analogous.  We have $\sigma_1^{0,n} = \sigma_J$ for $J =\{  3, \ldots,2g \}$.  Write $S\setminus \sigma_1^{0,n}=S_L \cup S_R$ as above. Using the construction described in Case 1 (and in Section \ref{S:closed}) for the increasing sequence of intervals $J_1 \subset \ldots \subset J_{2g-3} = J$, where $J_k = \{2g,2g-1, \ldots, 2g-k\}$, we obtain a pants decomposition of $S_R$. Similarly, we argue as in Section \ref{S:tori} to construct a pants decomposition of $S_L \setminus N(\alpha_0^i \cup \alpha_1)$ (see also Remark \ref{R:comparetorus} above). In this way we have constructed a pants decomposition, consisting entirely of elements of $\XX$,  of the complement of  $N(\alpha_0^i\cup\alpha_1)$; see  Figure \ref{F:puncturedint}. As above, this shows that the intersection of $\alpha_0^i$ and $\alpha_1$ is $\XX$--detectable.
\end{proof}

\begin{figure}[htb]
\begin{center}
\includegraphics[width=2.5in]{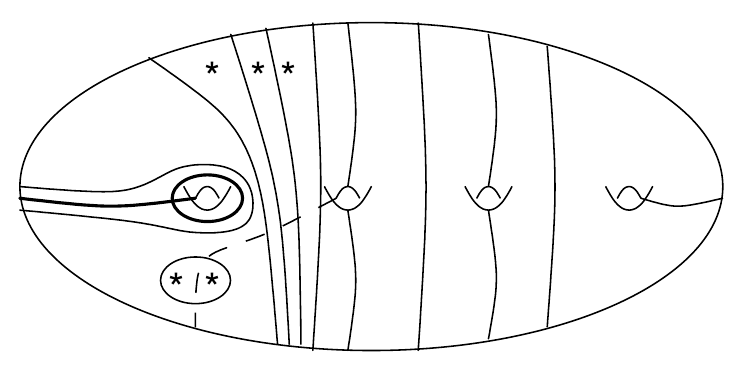} \caption{A pants decomposition $P$ of the complement of $\alpha_0^4 \cup \alpha_1$ in $S_{4,5}$ (in solid). The dashed curve $\beta_{0,1,2}^{4,-}$ is used to detect intersections with some of the curves in $P$.} \label{F:puncturedint}
\end{center}
\end{figure}

The above claim, together with Lemma \ref{L:minimallyfilling} applied to the appropriate four chain curves,  yield that $\phi(\alpha)$ and $\phi(\alpha')$ are Farey neighbors or nearly Farey neighbors.

\begin{claim}
Given $\alpha,\alpha' \in \CC$ with $i(\alpha,\alpha') = 1$, let $P_\alpha,P_{\alpha'}$ be the pants decompositions constructed in the proof of the claim above.  Then for every  $\gamma \in P_\alpha - (\alpha \cup \partial N(\alpha \cup \alpha'))$, there is $\gamma'\in\XX$ with $i(\gamma,\gamma') \neq 0$ $\XX$--detectable, and $i(\alpha,\gamma') = i(\alpha',\gamma') = 0$. 
\end{claim}
\begin{proof}[Proof of claim.]
Before starting we remark that, as in Section \ref{S:tori}, the intersection of any two outer curves is $\XX$--detectable; see Figure \ref{F:outerdetect}. 

As in the previous claim, there are two cases to consider.\\
\begin{figure}[htb]
\begin{center}
\includegraphics[width=2.5in]{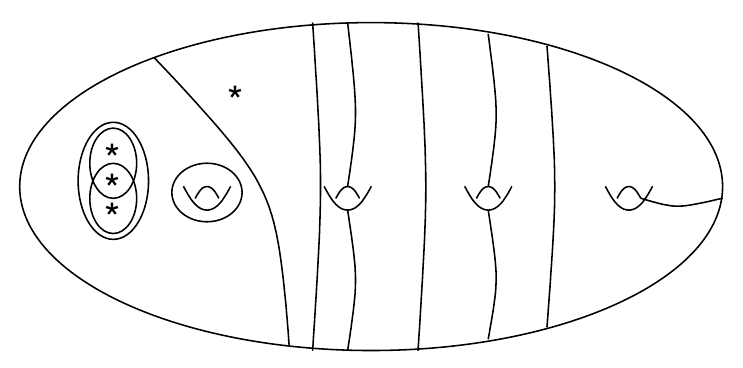} \caption{Detecting the intersection of two outer curves in  $S_{4,4}$ } \label{F:outerdetect}
\end{center}
\end{figure}

\noindent
{\bf Case 1.}  Neither $\alpha$ nor $\alpha'$ is in $\CC_f$.\smallskip

Assume first that $\alpha = \alpha_j$ and $\alpha'  = \alpha_{j+1}$ with $j$ even.  As in Section \ref{S:closed},  if $k\ge 3$ is odd, then $\sigma_{J_k}$ and $\alpha_{j+k + 1}$ have $\XX$-detectable intersection; similarly, the intersection of $\beta_{J_k}^\pm$ and $\alpha_{j+k+1}$ is $\XX$-detectable for $k$ even. 
The remaining curves in $P_\alpha \setminus \alpha$ are outer curves of the form $\epsilon^{i,n}$ with $i > 1$; note that $\epsilon^{0,i+1}$ has $\XX$-detectable intersection with $\epsilon^{i,n}$, and that $\epsilon^{0,i+1}$ is disjoint from $\alpha_j \cup \alpha_{j+1}$, as required.  For odd $j$, we argue similarly.\\

\noindent
{\bf Case 2.} One of $\alpha$ or $\alpha'$ is in $\CC_f$.\smallskip

Again without loss of generality, assume $\alpha = \alpha_0^i$ and $\alpha' = \alpha_1$, as the case $\alpha' = \alpha_{2g+1}$ is similar. Let $\gamma \in P_\alpha - (\alpha \cup \partial N(\alpha\cup \alpha'))$. Consider the separating curve $\sigma_1^{0,n}\in P_\alpha$, and write $S\setminus \sigma_1^{0,n} = S_L \cup S_R$ as above.
If $\gamma\subset S_L$, then the existence of the curve $\gamma'$ follows as in Section \ref{S:tori}. If $\gamma\subset S_R$ then we can use Case 1 to find the desired curve $\gamma'$. Finally, if $\gamma=\sigma_1^{0,n}$, we choose $\gamma'$ to be a curve in $\BB_T$; see the dashed curve in Figure \ref{F:puncturedint}. Observe that the intersection of $\sigma_1^{0,n}$ and $\gamma'$ is $\XX$-detectable.

\end{proof} 

From the second claim, it follows that for all $\alpha,\alpha' \in \CC$ with $i(\alpha,\alpha') = 1$, $N = N(\phi(\alpha)\cup\phi(\alpha'))$ has only one boundary component (that is, only one of the holes is a boundary component).  Suppose, for contradiction, that $N$ were homeomorphic to a sphere with four holes, three of which correspond to punctures of $S$ (compare Section \ref{S:tori}). 
In particular, each of $\phi(\alpha)$ and $\phi(\alpha')$ bound a disk containing exactly two punctures of $S$.  Now, any  $\alpha'' \in \CC$ can be connected to $\alpha$ by a sequence of curves in $\CC$ such that any two consecutive elements intersect once. By considering the $\phi$-image of the elements of such sequence, we obtain a sequence of curves connecting  $\phi(\alpha'')$ and $\phi(\alpha)$ such that any two consecutive elements are Farey neighbors or nearly Farey neighbors. Therefore, every element of $\phi(\CC)$ also bounds a disk containing exactly two punctures of $S$. However, this is impossible, since $S$ has $n$ punctures and $\CC$ contains $g+n$ pairwise disjoint curves. This contradiction proves  $i(\phi(\alpha),\phi(\alpha'))=1$, as desired.

It follows that there is a homeomorphism $h$ from a neighborhood of the curves in $\CC$ to a neighborhood of the curves in $\phi(\CC)$, with $h(\alpha_i) = \phi(\alpha_i)$ for all $i \neq 0$, and $\phi(\CC_f) = h(\CC_f)$.  Arguing as in Section \ref{S:tori}, we may use the curves in $\OO$ and $\SS_T$ to deduce that $\phi$ preserves the cyclic order of the curves $\CC_f$ (observe that the intersections between $\CC_f$ and $\OO$ are $\XX$--detectable, as are intersections between $\CC_f$ and the curves $\sigma_1^{i,j} \in \SS_T$).
Moreover, using $\XX$--detectable intersections between curves in $\OO$ and $\BB_T$ (e.g. the dashed curve in Figure \ref{F:puncturedint}), it follows that the ``extremal''  curves $\alpha_0^0$ and $\alpha_0^n$ are sent by $\phi$ to the extremal curves of $\phi(\CC_f)$.  Therefore, we may assume $h(\alpha_0^i) = \phi(\alpha_0^i)$ for all $i$.

Since $\phi$ is locally injective, it follows that $\phi(\alpha_0^i)$ and $\phi(\alpha_0^{i+1})$ bound a once punctured annulus for all $i = 0,\ldots,n-1$.  However, the boundary of the regular neighborhood of $\phi(\CC)$ has two components between $\phi(\alpha_0^i)$ and $\phi(\alpha_0^{i+1})$---one of them bounds a once-punctured disk, while the other bounds an unpunctured disk.   The same is true of the regular neighborhood of $\CC$.  We now explain how to choose $h$ so that the  boundary components of the neighborhood of $\CC$ bounding punctured disks are sent to boundary components bounding punctured disks, and hence $h$ can be extended to all of $S$.

Let $\Theta_o^+$ and $\Theta_o^-$ be the two connected components of the complement of the odd chain curves; up to relabeling, $\Theta_o^+$ is a sphere with $g+n+ 1$ holes, $n$ of which are punctures of $S$, whereas $\Theta_o^-$ is a sphere with $g+1$ holes. Note that up to isotopy of the neighborhood of $\CC$, $h$ is only well-defined up to an orientation-reversing involution swapping $\Theta_o^+$ and $\Theta_o^-$.   Denote by $\phi(\Theta)_o^\pm$ the two connected components of the $\phi$-image of the odd chain curves, noting that each of $\phi(\Theta)_o^\pm$ is a sphere with holes, some (possibly none) of which are punctures of $S$.  
However, since the union of the outer curves minus $\epsilon^{0,n}$ is connected, and the intersection of any two of them is $\XX$-detectable, and since these are all disjoint from the odd chain curves, it follows that every puncture of $S$ is contained in $\phi(\Theta)_o^+$, say.   

Therefore, precomposing with the homeomorphism swapping $\Theta_o^+$ and $\Theta_o^-$ if necessary (as we may), we may assume that $h$ sends the intersection of the neighborhood of $\CC$ with $\Theta_o^+$ into $\phi(\Theta)_o^+$.    It follows that $h$ extends to a homeomorphism $h:S \to S$ with $h(\alpha) = \phi(\alpha)$ for all $\alpha \in \CC$.  As in previous sections, it remains to show that $h$ agrees with $\phi$ on the rest of the elements of $\XX$. 

First, observe that $\gamma \in \OO_2$ is uniquely determined by an almost filling set of the form $A= \{\alpha_0^0,\ldots,\alpha_0^i, \alpha_0^{i+2},\ldots,\alpha_0^n, \alpha_1, \ldots, \alpha_{2g+1}\}\subset \CC$, and therefore $h(\gamma)=\phi(\gamma)$.  As in Section \ref{S:tori}, if $k\ge 3$ then every element of $\OO_k$ is uniquely determined by an almost filling set whose elements belong to $\CC \cup \OO_2 \cup \ldots \cup \OO_{k-1}$. It follows that $h|_\OO = \phi |_\OO$.   Similarly, any $\sigma \in \SS_0\cup\SS_T$ is uniquely determined by an almost filling set in $\CC \cup \OO$; compare with Section \ref{S:closed}. Therefore, we also have $h |_{\SS_0 \cup \SS_T} = \phi |_{\SS_0 \cup \SS_T}$.

Now,  a bounding pair $\beta_J^{i,j,+},\beta_J^{i,j,-}$ in $\BB_T$ is uniquely determined by the almost filling set $A=\{\alpha \in \CC \cup \OO | i(\alpha, \gamma)=0\} \subset \CC\cup \OO$, and therefore $\{h(\beta_J^{i,j,+}),h(\beta_J^{i,j,-})\} = \{\phi(\beta_J^{i,j,+}),\phi(\beta_J^{i,j,-}) \}$. Since $\beta_J^{i,j,+}$ can be distinguished from $\beta_J^{i,j,-}$ by the curves in $\CC_f$ it intersects, and since such intersections are $\XX$-detectable, it follows that $h(\beta_J^{i,j,+}) = \phi(\beta_J^{i,j,+})$ and $h(\beta_J^{i,j,-}) = \phi(\beta_J^{i,j,-})$. Thus $h|_{\BB_T} = \phi |_{ \BB_T}$. 

Similarly, a bounding pair $\beta_J^+,\beta_J^-$ in $\BB_0$ is  uniquely determined by an almost filling subset of  $\CC\cup \OO$, so that $\{ \phi(\beta_J^+),\phi(\beta_J^-)\} = \{h(\beta_J^+),h(\beta_J^-)\}$ and it again suffices to check that $h(\beta_J^+) =\phi(\beta_J^+)$.   Let  $\phi(\Theta)_e^\pm$ denote the two components of the complement of the even elements of $\phi(\CC_0)$, labeled such that $\phi(\alpha_0^0)\in \phi(\Theta)_e^+$.  Similar to the closed case, one can check that if $\phi(\beta_J^+) \subset \phi(\Theta)_e^+$ for some $J$, then this holds for all $J$ (one may verify that the union of $\{\beta_J^+\}$ is connected, with any two connected by a chain of $\XX$--detectable intersections).  Since $\beta_{1,2}^{1,+} = \beta_{\{0,1,2 \}}^+ \in \BB_T \cap \BB_0$, we have already shown that this holds for one curve, and therefore it holds for all curves.

Finally, any element of $\UU$ is uniquely determined by an almost filling set in $\CC\cup\OO\cup\SS_0\cup\BB_0$, compare with Section \ref{S:closed}. Therefore  $h|_{\UU} = \phi |_{ \UU}$.  This finishes the proof of Theorem \ref{T:punctured}
\end{proof}


\begin{thebibliography}{99}

\bibitem{Behrstock-Margalit} J. Behrstock, D. Margalit, Curve complexes and finite index subgroups of mapping class groups, {\em Geometriae Dedicata} {{118}}(1): 71-85, 2006.

\bibitem{Broaddus} N. Broaddus, {\em Homology of the curve complex and the Steinberg
module of the mapping class group}, to appear in Duke Math. Journal.

\bibitem{Harer} J. Harer, {\em The virtual cohomological dimension of the mapping class group of an orientable surface}, Invent. Math.
84 (1986).

\bibitem{Irmak} E. Irmak, Superinjective simplicial maps of complexes of curves and injective homomorphisms of subgroups of mapping class groups, {\em{Topology}} {{43}} (2004), No.3. 


\bibitem{Ivanov} N. V. Ivanov, Automorphism of complexes of curves and of Teichm\"uller spaces.  {\em{Internat. Math. Res. Notices}},  {14} (1997), 651--666.


\bibitem{Korkmaz} M. Korkmaz, Automorphisms of complexes of curves on punctured spheres and on punctured tori.  {\em{Topology Appl.}},  {95}  (1999),  no. 2, 85--111.

\bibitem{Lee} C. W. Lee, {\em The associahedron and triangulations of the $n$-gon}, European J. Combin. 10 (1989). 


\bibitem{Louder} L. Louder, personal communication.

\bibitem{Luo} F. Luo, Automorphisms of the complex of curves.  {\em{Topology}},  {39}  (2000),  no. 2, 283--298.

\bibitem{Margalit} D. Margalit, Automorphisms of the pants complex, {\em{Duke Mathematical Journal}}, {{121}} (2004), no. 3, 457--479.



\bibitem{Shackleton} K. J. Shackleton, Combinatorial rigidity in curve complexes and mapping class groups, {\em{Pacific Journal of Mathematics}}, {{230}}, No. 1, 2007





\end{thebibliography}
\end{document}